\documentclass[10pt]{amsart}
\usepackage{amssymb,amsbsy,amsmath,amsfonts,times}
\usepackage{latexsym,euscript,exscale}

\usepackage{helvet}

\title{Tightness of Banach spaces and Baire category}
\author{Valentin Ferenczi  and Gilles Godefroy}
\address{Instituto de Matem\'atica e Estat\'istica \\
 Universidade de S\~ao Paulo \\
rua do Mat\~ao 1010 \\
Cidade Universit\'aria \\
05508-90 S\~ao Paulo, SP \\
Brazil.}
\email{ferenczi@ime.usp.br}
\address{Institut de Math\'ematiques de Jussieu \\
4 place Jussieu\\
75252 Paris Cedex 05\\
France}
\email{godefroy@math.jussieu.fr}
\urladdr{http://www.math.jussieu.fr/~godefroy}

\date{}
\newcommand {\N}{\mathbb N}
\newcommand {\Q}{\mathbb Q}

\newcommand{\embed}{\sqsubseteq}

\newcommand {\del}{ \; \big| \;}

\newtheorem{thm}{Theorem}[section]
\newtheorem{cor}[thm]{Corollary}
\newtheorem{lemme}[thm]{Lemma}
\newtheorem{prop} [thm] {Proposition}

\newtheorem{rem}[thm]{Remark}

\newtheorem{prob}[thm]{Problem}

\newtheorem{ex}[thm]{Example}

\usepackage{fouriernc}

\thanks{The authors wish to thank Christian Rosendal for many useful conversations.}

\begin{document}


\begin{abstract} We prove several dichotomies on linear embeddings between Banach spaces. Given an arbitrary Banach space $X$ with a basis, we show that the relations of isomorphism and bi-embedding are meager or co-meager on the Polish set of block-subspaces of $X$. We relate this result
with tightness and minimality of Banach spaces. Examples and open questions are provided.
\end{abstract}
\maketitle

\tableofcontents

\section{Introduction}
W.T. Gowers' fundamental results in geometry of Banach spaces \cite{Go,Go1} opened the way to a loose classification of Banach spaces up to subspaces, known as Gowers'  program (see \cite{FeR2}). We focus in this note on a specific question: how many subspaces - up to linear isomorphism - does a non-hilbertian Banach space contain? More precisely, this note gathers several observations in the spirit of previous work by C. Rosendal and the first-named author (\cite{Fe,FeR,FeR1,FeR2}). which were not spelled out before.
These remarks relate in particular the notion of tightness (from \cite{FeR2}) to Baire category arguments.

Our main results are dichotomies: Theorem \ref{2.1} is an embedding dichotomy into a Banach space with a basis. Theorem \ref{3.1} states that the relations of linear isomorphism and of bi-embeddability are meager or co-meager on the set $b(X)$ of block-subspaces of a space $X$ with a basis. Several examples are given and commented open questions conclude the note.

\

We use the classical notation and terminology for Banach spaces, as may be found in \cite{LT}. Our reference for topology and descriptive set theory results is \cite{Ke}. 

Specific pieces of notation are needed for block-bases and subspaces, and for these we follow \cite{Fe} and \cite{FeR3}. We differ, however, on the following: what is denoted $bb_d(X)$ there is denoted here $b(X)$. 

We recall what this notation means. Let $X$ is a Banach space equipped with a basis $(e_n)$. Given a field $Q$ of scalars, denote by $c_{00}(Q)$ the $Q$-vector space generated by the basis $(e_n)$.   We fix   a countable field $Q$ containing the set $\Q$ of rationals (or $\Q+i\Q$ in the complex case), and the norm of any
vector in $c_{00}(Q)$ - such a $Q$ is easily constructed by induction. Then let 
$Q_0$ be the set of normalized vectors of $c_{00}(Q)$.

We equip the countable set $Q_0$ with the discrete topology. The set $b(X)$ consists of all block-bases made of vectors in $Q_0$, while $b^{<\omega}(X)$ is the set of finite block-bases made of such vectors. The set $b(X)$ is a closed subset of $Q_0^{\omega}$, and thus it is a Polish space. 
Although the set $b(X)$ depends not only on $X$, but also on the choice of the basis $(e_n)$, there will be no ambiguity from this notation, since we shall always work with a fixed basis $(e_n)$ of $X$.

We recall that the support of a vector $x=\sum_n x_n e_n$ of $Q_0$ is the set $${\rm supp}\ x=\{n: x_n \neq 0\},$$ while the range of $x$ is the interval of integers $${\rm range}\ x=
[\min({\rm supp}\ x), \max({\rm supp}\ x)].$$

\section{Some topological lemmas}

We recall in this section some well-known results on Baire dichotomies. Our first lemma is  called the first topological 0-1 law in \cite{Ke} (Th. 8.46). It appears in \cite{G}, Lemma 2, but was certainly known earlier.

\begin{lemme}\label{1.1} Let $P$ be a Polish space, and $G$ be a group of homeomorphisms of $P$ such that for all $U, V$ non-empty open sets in $P$, there is $g \in G$ such that $g(U) \cap V \neq \emptyset$. Let $A \subset P$ with the Baire Property such that $g(A)=A$ for all $g \in G$. Then $A$ is meager or comeager.
\end{lemme}

\begin{proof} Let $B=P \setminus A$. If $A$ and $B$ are both non-meager, then there exist two non-empty open sets $U$ and $V$ such that $U \cap B$ and $V \cap A$ are both meager. Let $g \in G$ be such that the open set $W:=g(U) \cap V$ is non-empty. Since $g(U) \cap B=g(U \cap B)$, we have that $W \cap B$ and $W \cap A$ are both meager, and this is a contradiction.
\end{proof}

\begin{ex}\label{1.2} The relations $E_0$ and $E^{\prime}_0$. \end{ex} We see the Cantor set $2^{\omega}$ as the set of subsets of $\omega$, the set $2^{<\omega}$ as the set of finite subsets of $\omega$, and we define on $2^{\omega}$ the following relations.
\begin{enumerate}
\item $u E_0 v$ if there is $n \geq 0$ such that 
$$u \cap [n,+\infty)=v \cap [n, +\infty),$$

\item $u E_0^{\prime} v$ if there is $n \geq 0$ such that
$$u \cap [n,+\infty)=v \cap [n, +\infty)$$ and
$$|u \cap [0,n-1]|=|v \cap [0,n-1]|.$$
\end{enumerate}
Then the equivalence classes for $E_0$ or $E_0^{\prime}$ are orbits of groups of homeomorphisms, namely, for $E_0$,
$$G_0=\{(u \Delta .), u \in 2^{\omega}\},$$
and for $E_0^{\prime}$, the group $G_0^{\prime}$ of permutations of $\omega$ with finite support. Therefore any subset of $2^{\omega}$ with the Baire property which is $E_0-$, or (merely) $E_0^{\prime}$-saturated, is meager or comeager.

\

Our second lemma is a standard compactness argument ( see \cite{FeR}, Lemma 7).

\begin{lemme}\label{1.3} Let $A$ be a subset of $2^{\omega}$. The following assertions are equivalent:
\begin{enumerate}
\item $A$ is comeager,
\item there is a sequence $I_0<I_1<I_2<\cdots$ of successive subsets of $\omega$, and $a_n \subset I_n$, such that for any $u \in 2^{\omega}$, if the set $\{n: u \cap I_n=a_n\}$ is infinite, then $u \in A$.
\end{enumerate}
\end{lemme}

\begin{proof} For the reverse implication, just note that
$$O_n=\{u \in 2^{\omega} \del \exists k \geq n, u \cap I_k=a_k\}$$
is a dense open set of $2^{\omega}$ for any $n \geq 1$, and that
$$\cap_{n \geq 1}O_n \subset A.$$

For the direct implication, assuming $A$ is comeager, we write
$$2^{\omega} \setminus A \subset \cup_{n \geq 0}F_n,$$ where each $F_n$ is closed with empty interior. The "compactness" we use is actually the trivial fact that a set with two points is compact. An easy induction
argument 
provides $I_n$ and $a_n$ such that
$$u \cap I_n=a_n \Rightarrow u \notin \cup_{i<n}F_i.$$
If $u \in F_k$, then $u \cap I_n \neq a_n$ for all $n>k$, and the conclusion follows.
\end{proof}

 It results from the proof that we can assume without loss of generality that the $I_k$'s constitute a partition of $\omega$ into intervals.

\begin{cor}\label{1.5} Let $A$ be a subset of $\omega$ such that:
$$u \in A, u \subset v \Rightarrow v \in A.$$
Then
\begin{itemize}
\item[(a)] $A$ is meager if and only if there exist $I_0<I_1<I_2<\cdots$ such that
$$u \in A \Rightarrow \{n; u \cap I_n=\emptyset\} {\rm\ is\ finite}.$$
\item[(b)] $A$ is comeager if and only if there exist $I_0<I_1<I_2<\cdots$ such that
if $u$ contains infinitely many $I_n$'s, then $u \in A$.
\end{itemize}
\end{cor}

This corollary easily follows from Lemma \ref{1.3}. We note that (a) is shown in
\cite{GT}, where it is applied to filters and simply additive measures on $\omega$.

\

 There is a counterpart of Lemma \ref{1.3} for the space $b(X)$ of block-bases of $X$ \cite{FeR1}, which we state below as Proposition 2.5. In this case, one uses compactness and not finiteness, so the general result involves $\epsilon$-nets. Here we shall only be interested in isomorphic properties of block-subspaces, which are preserved by small enough perturbations of the vectors of the basis, and therefore use a simpler form of the characterization of comeager sets of $b(X)$.

If $\tilde{x}$ is a finite block-sequence in $b^{<\omega}(X)$, we say that $z \in b(X)$
{\em passes through} $\tilde{x}$ if $z$ may be written as the concatenation
$$z=\tilde{y}^{\frown}\tilde{x}^\frown w$$
for some $\tilde{y} \in b^{<\omega}(X)$ and some $w \in b(X)$.

\begin{prop} Let $X$ be a Banach space with a basis $(e_n)$. Let $A \subset b(X)$ be such that
$$(y \in A \wedge \overline{\rm span}(y)=\overline{\rm span}(z)) \Rightarrow z \in A.$$
Then the following assertions are equivalent:
\begin{enumerate}
\item $A$ is comeager in $b(X)$,
\item there is a sequence $\tilde{x}_0<\tilde{x}_1<\tilde{x}_2<\cdots$ of successive elements  of $b^{<\omega}(X)$, such that for any $z \in b(X)$, if the set $\{n: z {\rm\ passes\ through\ } \tilde{x}_n\}$ is infinite, then $z \in A$.
\end{enumerate}
\end{prop}

\section{Application to embeddings of Banach spaces}

Let $X$ be a Banach space with a basis $(e_n)_n$. Following \cite{FeR2}, we say that
an (infinite-dimensional) space $Y$ is {\em tight in $X$}  if there is a sequence $I_0<I_1<I_2<\cdots$ of successive intervals such that for all infinite subset $J \subset \N$,
$$Y    \not\embed \overline{{\rm span}}[e_n, n \notin \cup_{j \in J}I_j],$$
where $\embed$ means "embeds isomorphically into".
We say that the space $X$ is {\em tight} if all infinite-dimensional spaces $Y$ are tight in $X$.

Of course these notions really depend on the choice of the basis $(e_n)_n$, so the notation is not exactly accurate, but this will not cause any problem since we shall consider only one choice of basis for $X$.

We recall that $X$ is {\em minimal} if every infinite dimensional subspace of $X$ contains an isomorphic copy of $X$. The main result of \cite{FeR2} asserts that every Banach space contains a tight subspace or a minimal subspace.

The notion of tightness can be linked with the Baire category statements of the previous section through the following results.

\begin{prop}\label{tightness} Let $X$ be a Banach space with a basis $(e_n)_n$ and let $Y$ be a Banach space.
Then the following are equivalent:
\begin{enumerate}
\item[(a)] $Y$ is tight in $X$,
\item[(b)] $E_Y:=\{u \in 2^{\omega}: Y \embed \overline{\rm span}[e_n, n \in u]\}$ is meager in $2^{\omega}$,
\item[(c)] ${\rm Emb}_Y:=\{z \in b(X): Y \embed \overline{\rm span}[z]\}$ is meager in $b(X)$.
\end{enumerate}
\end{prop}

\begin{proof} The implication $(a) \Rightarrow (b)$ is an immediate consequence of Corollary \ref{1.5} (a), applied to $E_Y=\{u \in 2^{\omega}: Y \embed \overline{\rm span}[e_n, n \in u]\}$.

To prove $(b) \Rightarrow (c)$, assume that $E_Y$ is meager.
Let $\phi: b(X) \rightarrow 2^\omega$ be defined by
$$\phi((z_n)_n)=\cup_n{{\rm supp}\ z_n}.$$
It is clear that $\overline{\rm span}\ z \subset \overline{\rm span}\{ e_i, i \in \phi(z)\}$, and 
therefore $${\rm Emb}_Y \subset \phi^{-1}(E_Y).$$
The map $\phi$ is continuous, and for any basic open set $U=N_{z_0,\ldots,z_n}$ of $b(X)$,
$$\phi(U)=N_{\cup_{i \leq n}{\rm supp}\ z_i} \setminus 2^{<\omega},$$
and therefore
$\overline{\phi(U)}=N_{\cup_{i \leq n}{\rm supp}\ z_i}$ is open. This easily implies that
$$A {\rm\ meager\ in\ } 2^\omega \Rightarrow \phi^{-1}(A) {\rm \ meager\ in\ } b(X).$$
If now (b) $E_Y$ is meager, then ${\rm Emb}_Y \subset \phi^{-1}(E_Y)$ is meager, and (c) holds.


To prove $(c) \Rightarrow (a)$, assume ${\rm Emb}_Y=\{z \in b(X): Y \embed \overline{\rm span}\ z\}$ is meager, and for all $j$, let $\tilde{x}_j \in b^{<\omega}(X)$ be such that if $z \in b(X)$, then
$$z {\rm\ passes\ through\ } \tilde{x}_j {\rm\ for\ infinitely\ many\ } j \Rightarrow Y \not\embed \overline{\rm span}\ z.$$
Let $I_j={\rm range\ }\tilde{x}_j$ for each $j$. Let $J$ be an  infinite subset of $\N$,
and consider $$W=\overline{{\rm span}}[e_n, n \notin \cup_{j \in J}I_j].$$
If $z$ is the concatenation (in the appropriate order) of the $e_n$'s for $n \notin \cup_{j \in J}I_j$ and of the $\tilde{x}_j$ for $j \in J$, then $z$ passes through $\tilde{x}_j$ for all $j \in J$ and therefore $Y$ does not embed into $\overline{\rm span}\ z$. Since $W \subset \overline{\rm span}\ z$, $Y$ does not embed into $W$. Since $J$ was arbitrary, we have proven that $Y$ is tight in $X$.
\end{proof}

\begin{thm}\label{2.1} Let $X$ be a Banach space with a basis $(e_n)$, and let $Y$ be a Banach space.
Then exactly one of the two following assertions holds:
\begin{itemize}
\item[(a)] there exists $I_0<I_1<I_2<\cdots$ such that for any $J \subset \omega$ infinite,
$$Y    \not\embed \overline{{\rm span}}[e_n, n \notin \cup_{j \in J}I_j],$$
\item[(b)] there exists  $J_0<J_1<J_2<\cdots$ such that for any $I \subset \omega$ infinite,
$$Y    \embed \overline{{\rm span}}[e_n, n \notin \cup_{i \in I}J_i].$$
\end{itemize}
\end{thm}

\begin{proof}
Recall that $$E_Y=\{u \in 2^{\omega}: Y \embed \overline{\rm span}[e_n, n \in u]\}.$$
It is easy to check that $E_Y$ is an analytic subset of $2^{\omega}$ and thus it is has the Baire property. Obviously, $u \in E_Y$ and $u \subset v$ implies that $v \in E_Y$.

if $u E_0^{\prime} v$ (see Example \ref{1.2}), then the closed linear spans of
$[e_n, n \in u]$ and $[e_n, n \in v]$ are isomorphic. Hence $E_Y$ is $E_0^{\prime}$-saturated and thus by Lemma 1.1 it is meager or comeager. The result now follows from Corollary \ref{1.5}. 

Note that for checking that (a) and (b) are mutually exclusive, it suffices to apply them with two infinite sets $I$ and $J$ 
such that $\big(\cup_{j \in J}I_j\big) \cap (\cup_{i \in I}J_i\big) = \emptyset.$
\end{proof}

\begin{ex}\label{2.2} Tightness and minimality. \end{ex}
A space $X$ is tight when (a) holds for any $Y$, or equivalently, for any block-subspace $Y=\overline{\rm span}\ y$ generated by some $y \in b(X)$.

 On the other hand, (b) holds for any  minimal subspace $Y$ of $X$: indeed if (a) holded, and if we picked a subspace $Z$ of $Y$ embedding isomorphically into $\overline{\rm span}[e_n; n \in \cup_{j \in K}I_j]$ for some $K \subset \omega$ coinfinite, then we would deduce from (a) that $Y$ does not embed into $Z$, contradicting minimality.

In particular we see that a tight space does not contain any minimal subspace. Also,
since every subspace of a minimal space is minimal, it follows that if $X$ is minimal, then (b) holds for every subspace $Y$ of $X$,

\begin{ex} Tightness with constants. \end{ex}

If there are successive subsets $I_j$ of $\N$ such that
for each $j$,
$$Y    \not\embed_j \overline{{\rm span}}[e_n, n \notin I_j],$$
where $\embed_j$ means "embeds with constant $j$", then we may use the $I_j$'s to prove that $Y$ is tight in $X$; we say in that case that 
$Y$ is tight in $X$ {\em with constants}. When all infinite-dimensional spaces are tight 
in $X$ with constants, then $X$ is said to be {\em  tight with constants}.
This notion was defined and studied in \cite{FeR2}; Tsirelson's space $T$ is the typical space satisfying tightness with constants. 

\

Defining for $j \in \N$, 
$$E_Y(j):=\{u \in 2^{\omega}: Y \embed_j \overline{\rm span}[e_n, n \in u]\},$$
we have of course $$E_Y=\cup_{j \in \N}E_Y(j).$$
The next proposition, a counterpart of Proposition  \ref{tightness} in the case of $j$-embeddings, shows that $Y$ being tight in $X$ with constants is equivalent to saying
that all $E_Y(j)$'s are nowhere dense in $2^{\omega}$. In other words we have in that case a very natural description of $E_Y$ as a countable union of nowhere dense sets. A similar result holds for ${\rm Emb}_Y$.

\begin{prop}\label{constants}
 Let $X$ be a Banach space with a basis $(e_n)_n$ and let $Y$ be a Banach space.
Then the following are equivalent:
\begin{enumerate}
\item[(a)] $Y$ is tight in $X$ with constants,
\item[(b)] $E_Y(j):=\{u \in 2^{\omega}: Y \embed_j \overline{\rm span}[e_n, n \in u]\}$ is nowhere dense in $2^{\omega}$ for all $j \in \N$,
\item[(c)] ${\rm Emb}_Y(j):=\{z \in b(X): Y \embed_j \overline{\rm span}[z]\}$ is nowhere dense in $b(X)$ for all $j \in \N$.
\end{enumerate}
\end{prop}

\begin{proof} $(a) \Rightarrow (b)$: let $(I_j)$ be successive such that for each $j$,
$Y \not\embed_j [e_n, n \notin I_j]$.
This means that
$$E_Y(j) \subset \{u \in 2^{\omega} \del u \cap I_j \neq \emptyset\},$$
and since $E_Y(j) \subseteq E_Y(k)$ whenever $j \leq k$, that
$$E_Y(j) \subset \bigcap_{k \geq j}\{u \in 2^{\omega} \del u \cap I_k \neq \emptyset\}.$$
The set on the right hand side of this inclusion is closed with empty interior, so $E_Y(j)$ is nowhere dense for all $j$.

To prove $(b) \Rightarrow (c)$, let $\phi: b(X) \rightarrow 2^\omega$ be defined as in Proposition \ref{tightness} by
$\phi((z_n)_n)=\cup_n{{\rm supp}\ z_n},$ 
therefore for $j \in \N$, $${\rm Emb}_Y(j) \subset \phi^{-1}(E_Y(j)).$$
Since the map $\phi$ is continuous, and for any basic open set $U$ of $b(X)$,
$\overline{\phi(U)}$ is open, it follows that
$$E_Y(j)\ {\rm nowhere\ dense\ } \Rightarrow \phi^{-1}(E_Y(j))
{\rm\ nowhere\ dense\ }\Rightarrow {\rm Emb}_Y(j) {\rm\ nowhere\ dense}.$$

To prove $(c) \Rightarrow (a)$, assume ${\rm Emb}_Y(j)=\{z \in b(X): Y \embed \overline{\rm span}\ z\}$ is nowhere dense for all $j \in \N$.
We may use induction to find successive $\tilde{x}_j \in b^{\omega}(X)$ so that if
$I_j$ denotes ${\rm range\ }\tilde{x}_j$ and 
$n_j:=\max\ I_j$, then
$$N_{(e_0,e_1,\ldots,e_{n_{j-1}})^\frown
\tilde{x}_j} \cap {\rm Emb}_Y(j) = \emptyset$$ for all $j$.
We may assume the $I_j$ form a partition of $\N$,  and this implies that
$Y$ does not embed with constant $j$ into $\overline{\rm span}[e_i, i \notin I_j]$. Therefore
(a) is proved.
\end{proof}

\

In the next section, we will display an embedding dichotomy similar to Theorem \ref{2.1} within the set $b(X)$ of block-subspaces of a given Banach space with a basis.

\section{Topological $0-1$-laws for the classical relations on $b(X)$}

Let $X$ be a Banach space with a basis $(e_n)$, and let $b(X)$ be the Polish space of its
block sequences. We denote by $\embed$ (resp. $\simeq$) the relation of embeddability (resp. isomorphism) between subspaces of $X$. We consider the following subsets of $b(X)^2$:
$${\rm Is}=\big\{(y,z) \in b(X)^2: \overline{\rm span}\ y \simeq \overline{\rm span}\ z\big\},$$
$${\rm Be}=\big\{(y,z) \in b(X)^2: \overline{\rm span}\ y \embed \overline{\rm span}\ z\ {\rm and\ }
\overline{\rm span}\ z \embed \overline{\rm span}\ y\big\},$$
$${\rm Emb}=\big\{(y,z) \in b(X)^2: \overline{\rm span}\ y \embed \overline{\rm span}\ z\big\}.$$
Obviously $${\rm Is} \subseteq {\rm Be} \subseteq {\rm Emb}.$$
The main result of this section is:

\begin{thm}\label{3.1} The relations ${\rm Is}$, ${\rm Be}$, and ${\rm Emb}$ are meager or comeager in the Polish space
$b(X)^2$. \end{thm}

\begin{proof} Pick $\tilde{x}$ and $\tilde{y}$ in $b^{<\omega}(X)$ with same length, and denote 
by $T_{\tilde{x},\tilde{y}}$ the homeomorphism $T$ of $b(X)$ defined by
$$T(\tilde{x}^\frown z)=\tilde{y}^\frown z,$$
$$T(\tilde{y}^\frown z)=\tilde{x}^\frown z,$$
$$T(z)=z {\rm \ if\ } z \notin N(\tilde{x}) \cup N(\tilde{y}).$$
In other words, $T_{\tilde{x},\tilde{y}}$ substitutes $\tilde{x}$ to $\tilde{y}$ (and conversely) in
$N(\tilde{x}) \cup N(\tilde{y})$ and does nothing else.

Let $G$ be the group of homeomorphisms of $b(X)^2$ generated by the maps
$(T_{\tilde{x},\tilde{y}},T_{\tilde{u},\tilde{v}})$, where $(\tilde{x},\tilde{y})$
and $(\tilde{u},\tilde{v})$  are arbitrary pairs of elements of $b^{<\omega}(X)$ with same length.
It is easily seen that the $G$-orbit of any point $(x,y) \in b(X)^2$ is dense. Moreover, one clearly has
$$\overline{\rm span}\ T_{\tilde{x},\tilde{y}}(u)\simeq \overline{\rm span}\ u$$
for all $u \in b(X)$, and it follows that the sets ${\rm Is}$, ${\rm Be}$ and ${\rm Emb}$ are $G$-invariant. They are clearly analytic, hence have the Baire Property, and Lemma \ref{1.1} concludes the proof.
\end{proof}

\begin{rem}\label{3.2} Continuum of non-isomorphic subspaces: \end{rem} The
 Kuratowski-Mycielski theorem (see (19.1) in \cite{Ke}) asserts that if a relation $R$ is meager in the perfect Polish space $b(X)^2$, then there is an homeomorphic copy $K$ of the Cantor set in $b(X)$ such that $\neg(x R y)$ for all $x \neq y$ in $K$. Hence if ${\rm Is}$ is meager, the space $X$ contains a continuum of non-isomorphic subspaces.

\

 If $j \in \N$, we denote ${\rm Emb}(j)=\big\{(y,z) \in b(x)^2 \del \overline{\rm span}\ y \embed_j
\overline{\rm span}\ z\big\}$, and observe that
$${\rm Emb}=\cup_{j \in \N}{\rm Emb}(j).$$
The next observation makes a link with the previous section.

\begin{prop}\label{3.3} Let $X$ be a space with a basis $(e_n)$. The following hold: \begin{itemize}
\item[(a)] If $X$ is tight, then the relation ${\rm Emb}$ is meager in $b(X)^2$.
\item[(b)] If $X$ is tight with constants, then the relation ${\rm Emb}(j)$ is nowhere dense in $b(X)^2$ for all $j \in \N$.
\end{itemize} \end{prop}

\begin{proof} (a) If $X$ is tight, then by Proposition \ref{tightness} (c), the set 
$${\rm Emb}_y=\{z \in b(X): (y,z) \in {\rm Emb}\}$$
is meager in $b(X)$ for any $y \in b(X)$. The Kuratowski-Ulam theorem (\cite{Ke} Th. 8.41) then shows that ${\rm Emb}$ is meager since all its fibers are. 

(b) If $X$ is tight by constants, then in particular, by  \cite{FeR2} Proposition 4.1, we may find a successive sequence $(I_j)$ of intervals such that for all $j$, $$\overline{\rm span}[e_i, i \in I_j] \not\embed_j \overline{\rm span}[e_i, i \notin I_j].$$
Fix $k \in \N$,  and given $\tilde{x},\tilde{y}$  in $b^{<\omega}(X)$, pick 
$j \geq k$ so that $I_j$ is supported after $\tilde{x}$ and $\tilde{y}$.
Then it follows that
$$\big(N(\tilde{x}^{\frown}(e_i)_{i \in I_j}) \times N(\tilde{y}^\frown e_{1+\max I_j})\big) \cap {\rm Emb}(j)=\emptyset.$$
Since ${\rm Emb}(k) \subset  {\rm Emb}(j)$ we deduce that
$N(\tilde{x}) \times N(\tilde{y})$ contains an open set which is disjoint from ${\rm Emb}(k)$.
Since $\tilde{x},\tilde{y}$ were arbitrary, this means that ${\rm Emb}(k)$ is nowhere dense.
 \end{proof}

Observe that from Proposition \ref{tightness} (b), we may deduce equivalently to (a) that if $X$ is tight, then
the set
$$\{(y,u) \in b(X)\times 2^\omega: \overline{\rm span}\ y \embed \overline{\rm span}[e_n, n \in u]\}$$
is meager in $b(X)\times 2^\omega.$

\

It follows from Proposition \ref{3.3} and the Kuratowski-Mycielski theorem that every tight space contains a continuum of subspaces which do not embed into each other. This also follows from 
(\cite{FeR2} Th.7.3).

In fact, this argument goes beyond the case of tight spaces, since we have:

\begin{ex} A  space with an unconditional basis which is not tight, although ${\rm Emb}$ is meager.\end{ex}

\begin{proof} Let $G_u$ be Gowers' "tight by support" space, that is, such that all disjointly supported subspaces on its canonical basis $(u_n)$ are totally incomparable \cite{Go}. Let $(f_n)$ be the canonical basis of $\ell_2$. We consider $X=G_u \oplus \ell_2$, equipped with the basis $(u_0,f_0,u_1,f_1,\ldots)$.

By the remarks of Example \ref{2.2}, (b) of Theorem \ref{2.1} holds for $Y=\ell_2$. Therefore
(a) does not hold for this choice of $Y$ and therefore $X$ is not tight.

To prove that ${\rm Emb}$ is meager, it is enough by the Kuratowski-Ulam theorem to prove that for $y$ in a comeager subset of $b(X)$, the set
${\rm Emb}_y$ is meager in $b(X)$, or equivalently by Proposition \ref{tightness}, that  the set
$$E_Y=\{u \in 2^{\omega}: Y \embed \overline{\rm span}[e_n, n \in u]\}$$
is meager (where $Y$ denotes $\overline{\rm span}\ y$).
Let therefore $\Omega$ be the comeager set of all $y \in b(X)$ which pass through infinitely many $u_n$'s. We claim that for $y \in \Omega$, the set $E_Y$ is meager in $2^{\omega}$.

We may and do assume that $y$ is a subsequence of $(u_n)$. If $E_Y$ is not meager, then it is comeager, and therefore by Corollary \ref{1.5}, there exist $I_0<I_1<I_2<\cdots$ such that if $u$ contains infinitely many $I_n$'s, then $u \in E_Y$. In other words, if $u$ contains
infinitely many $I_n$'s, then
$$Y \embed \overline{\rm span}[e_n, n \in u].$$
Passing to a subsequence of $y$ whose supports on $(e_i)$ are disjoint from infinitely many
$I_n$'s, and letting $u$ be the union of these $I_n$'s, we may therefore assume that
$$({\rm supp}\ y) \cap u=\emptyset.$$
This implies that $Y$ embeds into a direct sum $\ell_2 \oplus Z$, where $Z$ is a subspace of $G_u$ which is disjointly supported from $Y$. On the other hand, since $G_u$ is tight by support, $Y$ is totally incomparable with $\ell_2$ and with $Z$, therefore (\cite{LT} Prop. 2.c.5) every operator from $W$ to $\ell_2 \oplus Z$ is strictly singular, which is a contradiction.
\end{proof}

Hence the largest relation ${\rm Emb}$ can be meager for spaces which are not tight. On the other hand, the relation ${\rm Be}$ - and thus the relation ${\rm Emb}$ - is trivial for minimal spaces, and hence it is of course comeager. In the next section we shall see that the converse is false, even for the smallest relation ${\rm Is}$. We shall also show that ${\rm Is}$ may be meager for minimal spaces.

\section{Some more examples}

We start by giving two non-minimal examples of spaces for which ${\rm Is}$ is comeager. The first is an easily defined infinite $\ell_2$-sum which is not minimal. The second is more involved and does not even contain a minimal subspace.

\begin{ex}\label{ex1} An $\ell_2$-sum with an unconditional basis which is not minimal, but such that ${\rm Is}$ is comeager. \end{ex}

\begin{proof} We fix $p \neq 2$ and let $X=\Big(\sum_n \oplus \ell_p^n\Big)_2$.
The space $X$ is not minimal, since it contains $\ell_2$ but does not embed into $\ell_2$.
On the other hand it can be shown - using e.g. the arguments from \cite{LT}, Prop. 1.g.4 -
that if $z \in b(X)$, then $\overline{\rm span}\ z$ is isomorphic to $\ell_2$ or to $X$.
If $b(X)= A \cup B$ is the partition of $b(X)$ into the two corresponding ${\rm Is}$-classes,
we deduce that $A$ or $B$ is non-meager. Hence ${\rm Is}$ is non-meager and therefore comeager
by Theorem \ref{3.1} - equivalently, $A$ or $B$ is comeager in $b(X)$. \end{proof}

The comeager class in Example \ref{ex1} is actually the class of $X$. This follows for instance from the next observation. 

\begin{rem} Let $X$ be a space with an unconditional basis $(e_n)$. 
If $$\big\{z \in b(X) \del \overline{\rm span}\ z \simeq \ell_2\big\}$$
is comeager, then $X$ is isomorphic to $\ell_2$.  \end{rem}

\begin{proof} Assuming
$A=\{z \in b(X): \overline{\rm span}\ z \simeq \ell_2\}$ is comeager in $b(X)$, let $\tilde{x}_n \in b^{<\omega}(X)$ be successive such that if $z$ passes
through infinitely many $\tilde{x}_n$'s, then $\overline{\rm span}\ z$ is isomorphic to $\ell_2$.
W may assume that the intervals $I_n={\rm range}\ \tilde{x}_n$ form a partition of $\omega$.
Then the concatenation of the $\tilde{x}_{2n}$'s and of the $e_i$ for $i \in \cup_n I_{2n+1}$ is in $A$, from which it follows that
$$\overline{\rm span}[e_i, i \in \cup_n I_{2n+1}]$$
embeds into $\ell_2$ and therefore is isomorphic to $\ell_2$.
The same holds for
$$\overline{\rm span}[e_i, i \in \cup_n I_{2n}].$$
By unconditionality of $(e_n)$, it follows that $X$ is isomorphic to $\ell_2$. \end{proof}

\

For the next two examples we shall make use of several properties of Tsirelson's space $T$, its dual or its $2$-convexification $T^{(2)}$; all may be found in \cite{CS}. We shall also use the result from \cite{FeR2} stating that $T$ and $T^{(2)}$ are tight. 

Recall that two bases $(e_n)$ and $(f_n)$ are equivalent when the map defined by $T(e_n)=f_n$
 for all
 $n$ extends to a linear isomorphism of the closed linear spans of $(e_n)$ and $(f_n)$. A basis is subsymmetric if it is unconditional and equivalent to all its subsequences, and symmetric when it is equivalent to all its permutations.

\begin{lemme}\label{sub} Let $X$ be a space with a subsymmetric basis $(e_n)$. Then ${\rm Is}$ is comeager.
\end{lemme}

\begin{proof}
Assume $(e_n)$ is subsymmetric.   If $x=\sum_{i \in {\rm supp}(x)} x_i e_i$ and 
$y=\sum_{j \in {\rm supp}(y)} y_j e_j$ are finitely supported, we say that they have same distribution if
there is an order preserving  bijection $\sigma$ between ${\rm supp}(x)$ and ${\rm supp}(y)$ such that $y_{\sigma(i)}=x_i$ for all $i$.
Note that for vectors of $Q_0$, there are only countably many possible distributions, which we denote by $\{d_k, k \geq 1\}$. 
Let
 $$A=\big\{(z_n)_n \in b(X) \del \forall k \geq 1,
\ z_n{\rm \ has\ distribution\ } d_k {\rm\ for\ infinitely\ many\ }n{\rm 's}\ \big\}.$$ We claim that $A$ is comeager and contained in a ${\rm Is}$-class in $b(X)$. Then it follows immediately that ${\rm Is}$ is comeager.

To prove the second part of the claim, note that if $y,z$ belong to $A$, then one easily constructs by induction a subsequence $(z_{n_i})_i$ of $z$ such that each $z_{n_i}$ has the same distribution as $y_i$. By subsymmetry of the basis, it follows that the subsequence $(z_{n_i})_i$  is equivalent to $y$. Likewise $y$ is equivalent to a subsequence of $z$. Since both are unconditional, it follows by the Schroeder-Bernstein property for unconditional sequences (first proved by Mityagin    \cite{M}) that $z$ is equivalent to a permutation of $y$ and therefore that $\overline{\rm span}\ y \simeq
\overline{\rm span}\ z$. So $A$ is contained in a single ${\rm Is}$-class.

Finally to prove the first part of the claim, let $(\tilde{x}_n)_n$ be successive elements of  $b^{<\omega}(X)$, such that each $\tilde{x}_n$ is a sequence of $n$ vectors of respective distributions $d_1,d_2,\ldots,d_n$. 
Let $C$ be the comeager set of all $z$ in $b(X)$ which pass through infinitely many of the $\tilde{x}_n$'s. It is clear that any $z \in C$ contains infinitely many terms of distribution $d_k$ for each $k \geq 1$. Therefore $C \subset A$ and $A$ is comeager.
\end{proof}

\begin{ex}\label{S} A space without minimal subspaces, although ${\rm Is}$ is comeager. \end{ex}

\begin{proof} Let $X=S(T^{(2)})$, the symmetrization of the $2$-complexification of Tsirelson's space. The canonical basis of $X$ is symmetric, so ${\rm Is}(X)$ is comeager by Lemma \ref{sub}. On the other hand, by \cite{CS} Notes and Remarks 7) a) p.118, every subspace $Y$ of $X$ contains an isomorphic copy of a subspace of $T^{(2)}$. Since $T^{(2)}$ is tight, it contains no minimal subspace, which implies that $Y$ cannot be minimal.
\end{proof}

We note at this point that the spaces for which ${\rm Is}$ is comeager are those for which the existence of  a
continuum of non-isomorphic subspaces remains to be shown - and is still open in some simple cases, such as $\ell_p$ for 
$2<p<+\infty$.

Conversely to Example \ref{S}, the relation ${\rm Is}$ may be meager even for minimal spaces:

\begin{ex}\label{T} A space  which is minimal although ${\rm Is}$ is meager. \end{ex}

\begin{proof} We shall consider the space $T^*$, which is minimal by \cite{CS}, and prove that ${\rm Is}$ is meager
on $b(T^*)^2$. First we denote by $(e_n)$ the canonical basis of $T$ and  by $\simeq$ the relation on $2^{\omega}$ induced by isomorphism on $T$, i.e.
$$u \simeq v \Leftrightarrow \overline{\rm span}\ [e_i, i \in u] \simeq
\overline{\rm span}\ [e_i, i \in v].$$
We observe that  any $\simeq$-class on $2^{\omega}$ is meager.
Indeed if $u_0 \in 2^{\omega}$ and if $Y_0$ denotes $\overline{\rm span} [e_n, n \in u_0]$, then
$$\{u \in 2^{\omega}: u_0 \simeq u\} \subseteq 
\{u \in 2^{\omega}: Y_0 \subseteq \overline{\rm span} [e_n, n \in u]\}=E_{Y_0},$$
which is meager because $T$ is tight.

On the other hand, since the basis of $T$ is unconditional and $T$ is reflexive, we note that $\simeq$ is also the relation on $2^{\omega}$ induced by isomorphism on $T^*$, i.e.
$$u \simeq v \Leftrightarrow \overline{\rm span}\ [e_n^*, n \in u] \simeq
\overline{\rm span}\ [e_n^*, n \in v],$$
where $(e_n^*)$ is the canonical basis of $T^*$. So we may relate $(2^\omega,\simeq)$ to $(b(T^*),{\rm Is})$ as follows.
Let $\phi: b(T^*) \rightarrow 2^\omega$ be defined by
$$\phi((z_n)_n)=\cup_n \min({\rm supp}\ z_n).$$
By the properties of $T^*$ we have that the sequences $(z_n)$ and
$(e^*_{\min({\rm supp}\ z_n)})$ are equivalent and in particular span isomorphic subspaces
of $T^*$. In other words the spaces
$\overline{\rm span}\ [e_i^*, i \in \phi(z)]$ and $\overline{\rm span}\ z$ are isomorphic
for each $z \in b(T^*)$, and therefore 
$$(z,y) \in {\rm Is} \Leftrightarrow \overline{\rm span}\ z \simeq \overline{\rm span}\ y
\Leftrightarrow \phi(z) \simeq \phi(y) 
$$
for $z,y \in b(T^*)$.
Now if $A$ is any ${\rm Is}$-class on $b(T^*)$, then $\phi(A)$ is contained in a single $\simeq$-class on $2^{\omega}$, and therefore is meager. The map $\phi$ is continuous, and for any basic open set $N_{z_0,\ldots,z_n}$ of $b(X)$,
$\overline{\phi(N_{z_0,\ldots,z_n})}$ is a basic open set of $2^{\omega}$. It follows easily that
$A=\phi^{-1}(\phi(A))$ is meager. 
So all ${\rm Is}$-classes are meager in $b(T^*)$, and Kuratowski-Ulam theorem implies that
${\rm Is}$ is meager in $b(T^*)^2$. \end{proof}

Finally if we note 
$${\rm Emb}^*=\big\{ (y,z) \in b(X)^2:\ \overline{\rm span}\ z \embed \overline{\rm span}\ y\big\},$$
then of course $${\rm Be}={\rm Emb} \cap {\rm Emb}^*.$$
Since ${\rm Emb}^*$ is homeomorphic to ${\rm Emb}$, it follows that
 ${\rm Be}$ is comeager if and only if ${\rm Emb}$ is comeager. The above example shows however that ${\rm Is}$ can be meager while ${\rm Be}$ is comeager - and even equal to $b(X)^2$.

\section{Some open questions}
This work is motivated by the crucial problem of estimating the complexity of the linear isomorphism relation $\simeq$ on the set $SB(X)$ of subspaces of a Banach space $X$.
Gowers and Komorowski - Tomczak-Jaegermann' solution to Banach homogeneous space problem \cite{Go1} asserts that if $X \not\simeq \ell_2$, then $SB(X)$ contains at least two classes, but it is not known if, for example, it necessarily contains infinitely many classes.

Following \cite{FeR1}, we say that a separable Banach space $X$ is {\em ergodic} if $E_0$ Borel reduces to $\simeq$ on $SB(X)$, i.e. if there is $$f:2^{\omega} \rightarrow SB(X)$$ a Borel map
(when $SB(X)$ is equipped with the natural Effros Borel structure, see \cite{Bos}), such that
$$u E_0 v \Leftrightarrow f(u) \simeq f(v).$$

It is shown in \cite{FeR2} Th. 7.3 that every tight space has a strong $E_0$-antichain and thus is in particular ergodic. it is interesting to notice that spaces which are "close to $\ell_2$" but not $\ell_2$ are ergodic: indeed \cite{An} weak Hilbert spaces and asymptotically hilbertian spaces non-isomorphic to $\ell_2$ are ergodic.
We recall that by Kuratowski - Mycielski, any space $X$ such that ${\rm Is}$ is meager in $b(X)^2$ contains a continuum of non-isomorphic subspaces. Actually a similar argument shows that if ${\rm Is}$ is meager then $E_0$ reduces to ${\rm Is}$, and thefore $X$ is ergodic (Proposition 7 of \cite{FeR1}).

The main conjecture, already stated in \cite{FeR1}, is therefore:

\begin{prob}\label{4.1} Let $X$ be a separable Banach space which is not isomorphic to $\ell_2$. Is $X$ ergodic? \end{prob}

A slightly weaker form of this conjecture would be to show that any $X \not\simeq \ell_2$ contains a continuum of non-isomorphic subspaces. This is not known for $\ell_p$, $2<p<\infty$, although it is known that $\ell_p$ contains uncountably many non-isomorphic subspaces. And since $b(\ell_p)$ consists of a single isomorphism class, one has to deal with the whole set $SB(\ell_p)$ of closed linear subspaces of $\ell_p$.

By a theorem of Silver \cite{Si}, every Borel - or more generally coanalytic - equivalence relation on a Polish space with uncountably many classes actually has a continuum of classes.
To show this for $\simeq$ on $SB(\ell_p)$, it would therefore be sufficient to answer positively the following question:

\begin{prob}\label{4.2} Is the isomorphism relation $\simeq$ a Borel subset of $SB(\ell_p)^2$
($1 \leq p <+\infty, p \neq 2$)? \end{prob}

Note that it is analytic in $SB(X)^2$ for any Banach space $X$, and it is known to be non-Borel if e.g. $X={\mathcal C}(\Delta)$ (\cite{Bos}). We conjecture - unfortunately - a negative answer to Problem \ref{4.2}.

\

Finally, in Example \ref{T}, it is known that $T^*$ it is not "block-minimal", meaning that it is not true that it embeds as a block-subspace of all its block-subspaces. So  arguably the minimality of $T^*$ does not have much to do with the structure of the relation of isomorphism between block-subspaces of $T^*$. In this direction, the following remains open:

\begin{prob} Find a space $X$ which embeds as a block-subspace of all its block-subspaces, but such that ${\rm Is}$ is meager. \end{prob}

\end{document}